\documentclass{article}
\usepackage{amsmath,amsthm}

\usepackage{graphicx}
\usepackage{xargs}
\usepackage[active]{srcltx}

\usepackage[shortlabels]{enumitem} %label option for enumerate \alph* \Alph* \roman* \Roman* or other... [label=$\bullet$], [label=\alph*]

\usepackage[utf8]{inputenc}
\usepackage[english]{babel}
\usepackage{latexsym}
\usepackage{amssymb}
\usepackage{upgreek}
\usepackage{bbm}
\usepackage{mathrsfs}
\usepackage[norelsize,ruled,vlined]{algorithm2e}
\usepackage{amsmath}
\usepackage[textwidth=4cm,textsize=footnotesize]{todonotes}
\usepackage{accents}
\usepackage{dsfont}

\usepackage{aliascnt}
\usepackage{cleveref}
\makeatletter
\newtheorem{theorem}{Theorem}
\crefname{theorem}{theorem}{Theorems}
\Crefname{Theorem}{Theorem}{Theorems}

\newaliascnt{lemma}{theorem}
\newtheorem{lemma}[lemma]{Lemma}
\aliascntresetthe{lemma}
\crefname{lemma}{lemma}{lemmas}
\Crefname{Lemma}{Lemma}{Lemmas}

\newaliascnt{corollary}{theorem}

\aliascntresetthe{corollary}
\crefname{corollary}{corollary}{corollaries}
\Crefname{Corollary}{Corollary}{Corollaries}

\newaliascnt{proposition}{theorem}

\aliascntresetthe{proposition}
\crefname{proposition}{proposition}{propositions}
\Crefname{Proposition}{Proposition}{Propositions}

\newaliascnt{definition}{theorem}

\aliascntresetthe{definition}
\crefname{definition}{definition}{definitions}
\Crefname{Definition}{Definition}{Definitions}

\newaliascnt{definitionProposition}{theorem}

\aliascntresetthe{definitionProposition}
\crefname{Proposition and Definition}{Proposition and Definition}{Proposition and Definition}
\Crefname{Proposition and Definition}{Proposition and Definition}{Proposition and Definition}

\newaliascnt{remark}{theorem}
\newtheorem{remark}[remark]{Remark}
\aliascntresetthe{remark}
\crefname{remark}{remark}{remarks}
\Crefname{Remark}{Remark}{Remarks}

\crefname{example}{example}{examples}
\Crefname{Example}{Example}{Examples}

\crefname{figure}{figure}{figures}
\Crefname{Figure}{Figure}{Figures}

\Crefname{assumption}{\textbf{H}\hspace{-3pt}}{\textbf{H}\hspace{-3pt}}
\crefname{assumption}{\textbf{H}}{\textbf{H}}

\Crefname{assumptionF}{\textbf{F}\hspace{-3pt}}{\textbf{F}\hspace{-3pt}}
\crefname{assumptionF}{\textbf{F}}{\textbf{F}}

\Crefname{assumptionB}{\textbf{B}\hspace{-3pt}}{\textbf{B}\hspace{-3pt}}
\crefname{assumptionB}{\textbf{B}}{\textbf{B}}

\Crefname{assumptionM}{\textbf{M}\hspace{-3pt}}{\textbf{M}\hspace{-3pt}}
\crefname{assumptionM}{\textbf{M}}{\textbf{M}}

\newcommand{\eqsp}{\;}

\newcommand{\bB}{\mathbb{B}}
\newcommand{\Dbb}{\mathbb{D}}
\newcommand{\Pbb}{\mathbb{P}}
\newcommand{\Ebb}{\mathbb{E}}
\newcommand{\Nbb}{\mathbb{N}}
\newcommand{\Zbb}{\mathbb{Z}}
\newcommand{\R}{\mathbb{R}}

\newcommand{\cB}{\mathsf{B}}
\newcommand{\cC}{\mathsf{C}}
\newcommand{\cF}{\mathscr{F}}
\newcommand{\cX}{\mathscr{X}}

\def\Rbb{\mathbb{R}}

\newcommand{\sX}{\mathsf{X}}

\newcommand\as{a.s. }

\def\tilde{\widetilde} 
\def\rme{\mathrm{e}}
\def\rmd{\mathrm{d}}
\def\Xsigma{\cX}
\def\Xset{\sX}
\newcommandx{\indi}[2][1=]{\mathbbm{1}^{#1}_{#2}}
\newcommand{\indiacc}[1]{\mathbbm{1}_{\{#1\}}}

\newcommandx\sequence[3][2=,3=]
{\ifthenelse{\equal{#3}{}}{\ensuremath{\{ #1_{#2}\}}}{\ensuremath{\{ #1_{#2}, \eqsp #2 \in #3 \}}}}
\newcommandx{\sequencen}[2][2=n\in\N]{\ensuremath{\{ #1, \eqsp #2 \}}}
\newcommandx\sequenceDouble[4][3=,4=]
{\ifthenelse{\equal{#3}{}}{\ensuremath{\{ (#1_{#3},#2_{#3}) \}}}{\ensuremath{\{  (#1_{#3},#2_{#3}), \eqsp #3 \in #4 \}}}}
\newcommandx{\sequencenDouble}[3][3=n\in\N]{\ensuremath{\{ (#1_{n},#2_{n}), \eqsp #3 \}}}

\newcommand{\lrb}[1]{\left[ #1 \right]}
\newcommand{\lrcb}[1]{\left\{ #1 \right\}}
\newcommandx{\chunk}[3][1=0,3=n-1]{{#2}_{#1}^{#3}}
\newcommandx{\tvdist}[3][1=]{\ensuremath{\mathrm{d}^{#1}_{\mathrm{TV}}}(#2,#3)}
\newcommand{\tvdistsym}{\ensuremath{\mathrm{d}_{\mathrm{TV}}}}

\newcommandx{\CPE}[3][1=]{{\mathbb E}_{#1}\left[\left. #2 \, \right| #3 \right]} %%%% esperance conditionnelle
\newcommand{\functionboundeddiff}[2]{\mathbb{BD}(#1,#2)}
\newcommandx\supnorm[2][1=]{| #2 |^{#1}_\infty}
\newcommand{\fracaa}[2]{#1 / #2}

\usepackage[toc,page]{appendix} %package utilisé pour l'appendice voir : https://texfaq.org/FAQ-appendix
\usepackage{tikz}
\begin{document}
\title{A quantitative Mc Diarmid's inequality for geometrically ergodic Markov chains}
%\title{Density estimation for RWRE : a Bayesian non-parametric approach}
\author{A. Havet, M. Lerasle, E. Moulines and E. Vernet}
\maketitle

\begin{abstract}
 We state and prove a quantitative version of the bounded difference inequality for geometrically ergodic Markov chains.
 Our proof uses the same martingale decomposition as \cite{MR3407208} but, compared to this paper, the exact coupling argument is modified to fill a gap between the strongly aperiodic case and the general aperiodic case.
\end{abstract}

\noindent
{\small {\bf Keywords:} Concentration inequalities ; Markov chains ; Geometric ergodicity ; Coupling.}

\noindent
{\small {\bf AMS MSC 2010:} 60J05; 60E15.}

\section{Introduction}
The purpose of this note is to establish a quantitative version of Mc Diarmid's inequality for geometrically ergodic Markov chains.
Let $X_0,\ldots,X_{n-1}$ denote independent random variables taking values in a measurable space $(\sX,\cX)$ and $c=(c_0,\ldots,c_{n-1})$ denote a vector of non-negative real numbers. A function  $f:\sX^n\to\Rbb$ satisfies the bounded difference inequality if for all $x=(x_0,\ldots,x_{n-1})$ and
$y=(y_0,\ldots,y_{n-1})\in \sX^{n}$, we have
\begin{equation}\label{def:BDP}
|f(x)-f(y)|\leqslant \sum_{i=0}^{n-1}c_i\indiacc{x_i\ne y_i}\eqsp.
\end{equation}
The bounded difference inequality, first established in \cite{MR1036755}, shows that for all $t > 0$,
\[
\Pbb\big(f(X_0,\ldots,X_{n-1})-\Ebb[f(X_0,\ldots,X_{n-1})]>t\big)\leqslant \rme^{-2t^2/\|c\|^2}\eqsp,
\]
where $\|c\|^2=\sum_{i=0}^{n-1}c_i^2$.
Several attempts have been made to extend this result to Markov chains.
In \cite{MR2424985}, the concentration of particular functionals of the form $f(x_0,\ldots,x_{n-1})=\sup_{g\in \cF}\sum_{i=0}^{n-1}g(x_i)$, for centered functions $g$ in a class $\cF$ is established.
The concentration of general functionals (satisfying \eqref{def:BDP}) of geometrically ergodic Markov chains was established in \cite{MR3407208}, where it is also proved that geometric ergodicity is a necessary assumption. However, the result in \cite{MR3407208} is not quantitative.
It states that for all geometrically recurrent set $C$, there exists a constant $\beta$, depending on $C$ such that for all $x \in C$ and $t > 0$,
\begin{equation}
\label{eq:mcdiarmid-markov}
\Pbb_x\big(f(X_0,\ldots,X_{n-1})-\Ebb_x[f(X_0,\ldots,X_{n-1})]>t\big)\leqslant \rme^{-\beta t^2/\|c\|^2}\eqsp,
\end{equation}
where for any $x\in\Xsigma$, $\Pbb_x$ is the distribution of the Markov chain $\{X_k\}_{k=0}^\infty$ starting from $x$ (see the precise definition below).
In many  applications, it is necessary to get the explicit dependence of the constant $\beta$ as a function of the set $C$.
In particular, this problem arises when establishing posterior concentration rates of Bayesian non-parametric estimators; see for example \cite{rousseau:2016,ghosal:vandervaart:2017} for recent accounts on this theory.
To extend these results to Markovian settings, the result of \cite{MR3407208} cannot be applied directly and a quantitative version of  \eqref{eq:mcdiarmid-markov} is required, where the dependence of $\beta$ on constants characterizing the mixing of the Markov chain is needed; see for example \cite{vernet:2015,lecorff:lerasle:vernet:2018}.

A quantitative version of Mc Diarmid's inequality for Markov chains was established in \cite{paulin:2015}, where the constant $\beta$ depends here explicitly on the mixing time of the chain.
The existence of finite mixing times requires \emph{uniform} ergodicity of the chain, see for example \cite[Section 3.3]{MR2095565}, an  assumption that typically fails when the chain takes value in general state spaces.
In this note, we prove an extension of Mc Diarmid's inequality to geometrically ergodic Markov chains.
Our proof is based on  \cite{MR3407208},
but avoids the use of \cite[Lemma~6]{MR3407208} which requires the construction of an exact coupling.
Exact coupling can actually be built in the strongly aperiodic case but there is a gap in the general aperiodic case.

The remaining of the paper is decomposed as follows, Section~\ref{sec:Setting} introduces formally the notations and the assumptions of the main result, which is stated and proved in Section~\ref{sec:Main}.

\section{Notations and assumptions}\label{sec:Setting}
Let $(\sX,\cX)$ be a measurable space.
We denote by $\tvdistsym$  the total variation distance between probability measures.
For any sequence $x=\sequence{x}[n][\Nbb]$ and any non-negative integers $a$ and $b$, with $a\leqslant b$, let $\chunk[a]{x}[b]=(x_a,x_{a+1},\ldots,x_b)$.
For any $n\geqslant 0$ and any vector $c=\chunk[0]{c}[n-1]\in \R^n$, let $\|c\|$ denote the Euclidean norm of $c$ and $\|c\|_\infty=\max_{0\leqslant i\leqslant n-1}|c_i|$ denote its sup-norm.

We denote by $(\sX^{\Zbb_+},\cX^{\otimes\Zbb_+}, (\cF_k)_{k\geqslant 0})$ the canonical filtered space,
$\{X_n\}_{n=0}^\infty$ the canonical process and $\theta: \sX^{\Zbb_+}\to\sX^{\Zbb_+}$  the shift operator on the canonical space defined, for any $x=(x_{n})_{n\geqslant 0}\in \sX^{\Zbb_+}$ by $\theta(x)\in \sX^{\Zbb_+}$, where, for any $n\geqslant 0$, $\theta(x)_n=x_{n+1}$.
Set $\theta_1=\theta$ and for $n\in \Nbb^*$, define inductively, $\theta_n=\theta_{n-1} \circ \theta$. We also need to define $\theta_\infty$. To this aim, fix an arbitrary $x^*\in \Xset$, we define $\theta_\infty:\Xset^\Nbb \to \Xset^\Nbb$ such that for $z=\sequence{z}[k][\Nbb] \in \Xset^\Nbb$, $\theta_\infty z\in \Xset^\Nbb$ is the constant sequence  $(\theta_\infty z)_k=x^*$ for all $k \in \Nbb$.

Let $P$ be a Markov kernel on $\sX \times \cX$.
For any probability measure $\xi$ on $(\Xset,\Xsigma)$, denote by $\Pbb_{\xi}$  the unique probability under which $(X_n)_{n\geqslant 0}$ is a Markov chain with Markov kernel $P$ and initial distribution $\xi$ and let $\Ebb_{\xi}$ denote the expectation under the distribution $\Pbb_{\xi}$.
Recall that $\cF_n$ denotes the $\sigma$-algebra generated by $X_0, \ldots,X_n$.
For any $x\in \sX$, let $\delta_x$ denote the Dirac mass at point $x$.
With some abuse of notation, we also denote $\Pbb_x$ (resp. $\Ebb_x$) instead of $\Pbb_{\delta_x}$ (resp. $\Ebb_{\delta_x}$).

For any $\cB\in \cX$ and any integer $i\geqslant 0$, let
\[
\tau_{\cB}^i=\inf\{n\geqslant i: X_n\in \cB\}=i+\tau_{\cB}^0\circ \theta^i
\qquad \text{and} \qquad
\sigma_{\cB}=\tau_{\cB}^1=1+\tau_{\cB}^0\circ \theta\eqsp.
\]
For $c = \chunk[0]{c}[n-1] \in \R_+^n$, we denote by  $\functionboundeddiff{\sX^n}{c}$  the set of measurable functions $f:\sX^n\to\R$ such that for all $x= (x_0,\dots,x_{n-1})$ and $y= (y_0,\dots,y_{n-1})$, $|f(x)-f(y)|\leqslant \sum_{i=0}^{n-1} c_i\indiacc{x_i\ne y_i}$
The main result is established under the following conditions.
\begin{itemize}
\item[{\bf H1}] The Markov kernel $P$ is irreducible and aperiodic, with unique invariant probability $\pi$.
\item[{\bf H2}] There exist a non-empty set $\cC \in \cX$ and two real numbers $u>1$ and $M>0$ such that
  \[\sup_{x\in \cC}\Ebb_x[u^{\sigma_{\cC}}]\leqslant M\eqsp.\]
\item[{\bf H3}] There exist $r\in (0,1)$ and $L\geqslant 1$ such that, for any $x$ in the set $\cC$ of {\bf H2} and any $n\geqslant 0$,
\[
\tvdist{\delta_xP^n}{\pi} \leqslant Lr^n\eqsp,
\]
where $\pi$ is the unique invariant measure granted in {\bf H1}.
\end{itemize}
When the Markov kernel $P$ is uniformly ergodic, then {\bf H3} holds with $\cC=\sX$. The following Lemma is a coupling result that replaces \cite[Lemma~6]{MR3407208}.
It is instrumental in the sequel.
\begin{lemma}\label{lem:DMPS18}
 For any probability measures $\xi$ and $\xi'$ on $(\sX,\cX)$, any $n\geqslant 1$, any $c\in \R_+^n$ and any $h\in \bB\Dbb(\sX^n,c)$,
\begin{equation*}\label{eq:Step1}
 |\Ebb_\xi[h(\chunk[0]{X}[n-1])]-\Ebb_{\xi'}[h(\chunk[0]{X}[n-1])]|\leqslant 2\sum_{i=0}^{n-1}c_i \tvdist{\xi P^i}{\xi'P^i}\eqsp.
\end{equation*}
\end{lemma}
\begin{remark}
 It is possible to avoid the factor $2$ in \eqref{eq:Step1} under additional technical conditions, for example, when there exists a maximal coupling for $(\Pbb_\xi,\Pbb_{\xi'})$, see \cite[Lemma 23.2.1]{douc:moulines:priouret:soulier:2018}.
\end{remark}
\begin{proof}
Fix an arbitrary
$x^* \in \Xset$. For $i \in \{1,\ldots,n-1\}$, we set $\bar{h}_i(\chunk[i]{x}[n-1])=h(x^*,\dots,x^*,\chunk[i]{x}[n-1])$. By convention,
we set $\bar{h}_n$ the constant function $\bar{h}_n=h(x^*,\ldots,x^*)$ and $\bar{h}_0=h$. With these
notations, we have the decomposition
\[
h(\chunk{x})=\sum_{i=0}^{n-1} \{\bar{h}_i(\chunk[i]{x}[n-1]) - \bar{h}_{i+1}(\chunk[i+1]{x}[n-1]) \} +\bar{h}_n \eqsp.
\]
For all $i \in \{0,\ldots,n-1\}$ and all $x_i \in \Xset$, let
\begin{align}
\bar{w}_i(x_i)&=\int \lrcb{\bar{h}_i(\chunk[i]{x}[n-1])-\bar{h}_{i+1}(\chunk[i+1]{x}[n-1])}  \prod_{\ell=i+1}^{n-1} P(x_{\ell-1}, \rmd x_{\ell})\eqsp, \nonumber \\
&=\int \lrcb{h(x^*,\ldots,x^*,\chunk[i]{x})-h(x^*,\ldots,x^*,\chunk[i+1]{x})}  \prod_{\ell=i+1}^{n-1} P(x_{\ell-1}, \rmd x_{\ell}) \eqsp. \label{eq:delta:diarmid:fond}
\end{align}
It is easily seen that $\CPE{\{\bar{h}_i(\chunk[i]{X}[n-1])-\bar{h}_{i+1}(\chunk[i+1]{X}[n-1])\}}{\cF_{i}}=\bar{w}_i(X_i)$, $\Pbb_\xi-\as$,  which implies that
\[
\Ebb_\xi \lrb{h(\chunk{X})}=\sum_{i=0}^{n-1} \xi P^i \bar{w}_i + \bar{h}_n \eqsp.
\]
Since  $h \in\functionboundeddiff{\Xset^n}{c}$, \eqref{eq:delta:diarmid:fond} shows that  $\supnorm{\bar{w}_i} \leq c_i$.
Therefore,
\[
|\Ebb_\xi \lrb{h(X^{n-1})}-\Ebb_{\xi'} \lrb{h(X^{n-1})} |\\
\leq \sum_{i=0}^{n-1} |\xi P^i \bar{w}_i-\xi' P^i \bar{w}_i| \leq 2 \sum_{i=0}^{n-1} c_{i}\tvdist{\xi P^{i}}{\xi' P^{i}} \eqsp.
\]
\end{proof}

\section{Main result}\label{sec:Main}
The main result of this paper  is the following quantitative version of Mac Diarmid's inequality for geometrically ergodic Markov chains.
\begin{theorem}\label{thm:ConcMarkQuant}
 Assume {\bf H1}, {\bf H2}, {\bf H3}. Let $n\geqslant 1$, $c\in \R^n$ and  $f\in \functionboundeddiff{\sX^n}{c}$. Then, for all $x\in \cC$ and $t > 0$,
 \[
\Pbb_x\big(f( \chunk[0]{X}[n-1])-\Ebb_x[f(\chunk[0]{X}[n-1])]>t\big)\leqslant \exp\bigg(-\frac{\beta t^2}{\|c\|^2}\bigg)\eqsp,
 \]
where $\beta$ is given by
 \[
 \beta=\frac{(1-r\vee u^{-1/4})^2}{16L}\bigg(\frac5{\log u}+4ML\bigg)^{-1}\eqsp.
 \]
\end{theorem}
\begin{proof}[Proof of \Cref{thm:ConcMarkQuant}]
Fix $c\in \Rbb^n$, $x\in\sX$ and $f\in \functionboundeddiff{\sX^n}{c}$.
Following \cite{MR3407208}, we decompose $f(\chunk[0]{X}[n-1])-\Ebb_x[f(\chunk[0]{X}[n-1])]$ into martingale increments by conditioning to the stopping times $\tau_{\cC}^i$, $i=0,\ldots, n-1$.
For any integer $i\in [0,n-1]$, define
\[
G_i=\Ebb_x\big[f(\chunk[0]{X}[n-1])|\cF_{\tau_{\cC}^i}\big]\eqsp.
\]
As $\tau_{\cC}^0=0$ $\Pbb_x$-a.s., it holds $\Ebb_x[f(\chunk[0]{X}[n-1])]=\Ebb_x[f(\chunk[0]{X}[n-1])|\cF_{\tau_{\cC}^0}]=G_0$.
Moreover, as $\tau_{\cC}^{n-1}\geqslant n-1$, it also holds $G_{n-1}=\Ebb_x[f(\chunk[0]{X}[n-1])|\cF_{\tau_{\cC}^{n-1}}]=f(\chunk[0]{X}[n-1])$.
Therefore, the difference $f(\chunk[0]{X}[n-1])-\Ebb_x[f(\chunk[0]{X}[n-1])]$ is decomposed into a sum of the martingale increments $G_{i+1}-G_i$ as follows
\begin{equation}\label{eq:LinkWithG}
 f(\chunk[0]{X}[n-1])-\Ebb_x[f(\chunk[0]{X}[n-1])]=G_{n-1}-G_0=\sum_{i=0}^{n-2}(G_{i+1}-G_i)\eqsp.
\end{equation}
The proof is now decomposed into three facts that aim at bounding the Laplace transform of $f(\chunk[0]{X}[n-1])-\Ebb_x[f(\chunk[0]{X}[n-1])]$.

\noindent
{\bf Fact 1.} \emph{For any $i\in \{1,\ldots,n-1\}$,}
\begin{equation}\label{eq:Fact1}
 G_i-G_{i-1}=(G_i-G_{i-1})\indiacc{\tau_{\cC}^{i-1}=i-1}\eqsp.
\end{equation}
\begin{proof}[Proof of Fact 1.] By definition $\tau_{\cC}^{i-1}\geqslant i-1$ and $\tau_{\cC}^{i-1}> i-1$ if and only if $\tau_{\cC}^{i-1}=\tau_{\cC}^{i}$.
 Therefore,
 \[
 G_i-G_{i-1}=(G_i-G_{i-1})\left(\indiacc{\tau_{\cC}^{i-1}=i-1}+\indiacc{\tau_{\cC}^{i-1}=\tau_{\cC}^{i}}\right)\eqsp.
 \]
To prove that $(G_i-G_{i-1})\indiacc{\tau_{\cC}^{i-1}=\tau_{\cC}^{i}}=0$, we decompose according to the values of $\tau_{\cC}^{i}$:
 \[
(G_i-G_{i-1})\indiacc{\tau_{\cC}^{i-1}=\tau_{\cC}^{i}}=\sum_{j\geqslant i} (G_i-G_{i-1})\indiacc{\tau_{\cC}^{i-1}=\tau_{\cC}^{i}=j}\eqsp.
 \]
Now, remark that, for any $i\geqslant 0$,
\begin{align}
 G_i\indiacc{\tau_{\cC}^{i}=j}=
\begin{cases}\label{eq:Gitaui=j}
 \Ebb_x\big[f(\chunk[0]{X}[n-1])|\cF_j\big]&\text{ if } j\leqslant n-2\eqsp,\\
 f(\chunk[0]{X}[n-1])&\text{ if } j\geqslant n-1\eqsp.
\end{cases}
\end{align}
Then, for any $j\geqslant i$,
\[
 G_i\indiacc{\tau_{\cC}^{i}=j}\indiacc{\tau_{\cC}^{i-1}=\tau_{\cC}^{i}}= G_{i-1}\indiacc{\tau_{\cC}^{i-1}=j}\indiacc{\tau_{\cC}^{i-1}=\tau_{\cC}^{i}}=G_{i-1}\indiacc{\tau_{\cC}^{i}=j}\indiacc{\tau_{\cC}^{i-1}=\tau_{\cC}^{i}}\eqsp.
\]
This proves Fact 1.
\end{proof}

Fact 2. bounds the increments $G_i-G_{i-1}$. The proof relies on the following lemma which is a consequence of the coupling result Lemma~\ref{lem:DMPS18}.
Define $g_{n-1}=g_{n-1,\pi}=f$ and, for any $i\in [0,n-2]$, let $g_i$ and $g_{i,\pi}$ denote the functions defined for any $\chunk[0]{x}[i]\in\sX^{i+1}$ by 
\begin{equation}
\label{eq:definition-g-i}
g_i(\chunk[0]{x}[i])=\Ebb_{x_i}[f(\chunk[0]{x}[i],\chunk[1]{X}[n-1-i])],\qquad g_{i,\pi}(\chunk[0]{x}[i])=\Ebb_{\pi}[f(\chunk[0]{x}[i],\chunk[1]{X}[n-1-i])]\eqsp.
\end{equation}

\begin{lemma}
\label{lem:gi-gipi}
 Assume {\bf H1}, {\bf H2}, {\bf H3}. 
 For any $i\in \{0,\ldots,n-1\}$ and $(\chunk[0]{x}[i-1],x_i)$ in $\sX^i\times\cC$,
\begin{equation}\label{eq:gi-gipi}
|g_i(\chunk[0]{x}[i])-g_{i,\pi}(\chunk[0]{x}[i])|\leqslant 2L\sum_{j=i+1}^{n-1}c_jr^{j-i}\eqsp.
\end{equation}
\end{lemma}
\begin{proof}
 Fix $i\in\{0,\ldots,n-1\}$ and $\chunk[0]{x}[i]\in \sX^{i+1}$.
As $f\in\functionboundeddiff{\sX^n}{c}$, the function $\tilde{f}_i:\chunk[1]{y}[n-1-i]\in \sX^{n-1-i}\mapsto f(\chunk[0]{x}[i],\chunk[1]{y}[n-1-i])\in \R$ satisfies
\[
|\tilde{f}_i(\chunk[1]{y}[n-1-i])-\tilde{f}_i(\chunk[1]{z}[n-1-i])|\leqslant \sum_{k=1}^{n-1-i}c_{i+k}\indiacc{y_k\ne z_k}\eqsp.
\]
Hence, $\tilde{f}_i\in\bB\Dbb(\sX^{n-1-i},c_{i+1:n-1})$. Applying Lemma~\ref{lem:DMPS18} to the function $h=\tilde{f}_i$ yields
\begin{align*}
 |g_i(\chunk[0]{x}[i])-g_{i,\pi}(\chunk[0]{x}[i])|&=|\Ebb_{x_i}[f(\chunk[0]{x}[i],\chunk[1]{X}[n-1-i])]-\Ebb_{\pi}[f(\chunk[0]{x}[i],\chunk[1]{X}[n-1-i])]|\\
 &=|\Ebb_{x_i}[\tilde{f}_i(\chunk[1]{X}[n-1-i])]-\Ebb_{\pi}[\tilde{f}_i(\chunk[1]{X}[n-1-i])]| \leqslant 2\sum_{j=i+1}^{n-1}c_j \tvdist{\delta_{x_i}P^j}{\pi}\eqsp.
\end{align*}
Inequality \eqref{eq:gi-gipi} follows from {\bf H3}.
\end{proof}

\noindent
{\bf Fact 2.} \emph{Let $\rho$ such that $r\leqslant \rho<1$ and $i\in\{1,\ldots,n-1\}$.
Then,}
\begin{align}
\label{eq:Gi-Gi-1} |G_i-G_{i-1}|&\leqslant C_1\|c\|_{\infty}\indiacc{\tau_{\cC}^{i-1}=i-1}\sigma_{\cC}\circ\theta^{i-1}\eqsp,\\
\label{eq:Gi-Gi-1Sq}  |G_i-G_{i-1}|^2&\leqslant C_2\indiacc{\tau_{\cC}^{i-1}=i-1}\frac1{\rho^{2\sigma_{\cC}\circ\theta^{i-1}}}\sum_{k=i}^{n-1}c_k^2\rho^{k-i}\eqsp.
\end{align}
\emph{where, $C_1=5L/(1-r)$ and $C_2=16L^2/(1-\rho)$.}
\begin{proof}[Proof of Fact 2.]
For any integer $i\in\{1,\ldots,n\}$, let
\[
G_{i,1}=\Ebb_x[f(\chunk[0]{X}[n-1])|\cF_{\tau_{\cC}^{i-1}}]\indiacc{\tau_{\cC}^{i-1}=i-1},\qquad G_{i,2}=\Ebb_x[f(\chunk[0]{X}[n-1])|\cF_{\tau_{\cC}^{i}}]\indiacc{\tau_{\cC}^{i-1}=i-1}\eqsp.
\]
From Fact 1., $G_i-G_{i-1}=G_{i,2}-G_{i,1}$.
By Markov's property, for any $i\in \{0,\ldots,n-1\}$ and  $x\in \sX$,
\[
\Ebb_x[f(\chunk[0]{X}[n-1])|\cF_i]=g_i(X_{0:i}),\qquad \Pbb_{x}-\text{a.s.}\eqsp.
\]
Now, let $R_{i,1}=g_{i-1}(\chunk[0]{X}[i-1])\indiacc{\tau_{\cC}^{i-1}=i-1}-g_{i-1,\pi}(\chunk[0]{X}[i-1])\indiacc{\tau_{\cC}^{i-1}=i-1}$.
We have
\begin{align}
\notag G_{i,1}&=\Ebb_x[f(\chunk[0]{X}[n-1])|\cF_{\tau_{\cC}^{i-1}}]\indiacc{\tau_{\cC}^{i-1}=i-1}=\Ebb_x[f(\chunk[0]{X}[n-1])|\cF_{i-1}]\indiacc{\tau_{\cC}^{i-1}=i-1}\\
\label{eq:Gi1} &=g_{i-1}(\chunk[0]{X}[i-1])\indiacc{\tau_{\cC}^{i-1}=i-1}=g_{i-1,\pi}(\chunk[0]{X}[i-1])\indiacc{\tau_{\cC}^{i-1}=i-1}+R_{i,1}\eqsp.
\end{align}
Moreover, as $\tau_{\cC}^i\geqslant i$, by \eqref{eq:Gitaui=j},
\begin{align}
\notag G_{i,2}&=\sum_{j\geqslant i}\Ebb_x[f(\chunk[0]{X}[n-1])|\cF_{\tau_{\cC}^{i}}]\indiacc{\tau_{\cC}^{i-1}=i-1}\indiacc{\tau_{\cC}^i=j}\\
%\notag &=\sum_{j=i}^{n-2}\Ebb_x[f(\chunk[0]{X}[n-1])|\cF_{j}]\indiacc{\tau_{\cC}^{i-1}=i-1}\indiacc{\tau_{\cC}^i=j}+f(\chunk[0]{X}[n-1])\indiacc{\tau_{\cC}^{i-1}=i-1}\indiacc{\tau_{\cC}^i\geqslant n-1}\\
\label{eq:Gi2} &=\sum_{j=i}^{n-2}g_j(\chunk[0]{X}[j])\indiacc{\tau_{\cC}^{i-1}=i-1,\tau_{\cC}^i=j}+f(\chunk[0]{X}[n-1])\indiacc{\tau_{\cC}^{i-1}=i-1,\tau_{\cC}^i\geqslant n-1}\eqsp.
\end{align}
Let $R_{i,2}=\sum_{j=i}^{n-2}(g_j(\chunk[0]{X}[j])-g_{j,\pi}(\chunk[0]{X}[j]))\indiacc{\tau_{\cC}^{i-1}=i-1,\tau_{\cC}^i=j}$.
From \eqref{eq:Gi1} and \eqref{eq:Gi2},
\begin{align}
\label{eq:Gi1-Gi2} |G_{i,2}-G_{i,1}|=|&R_{i,2}-R_{i,1}+\sum_{j=i}^{n-2}(g_{j,\pi}(\chunk[0]{X}[j])-g_{i-1,\pi}(\chunk[0]{X}[i-1]))\indiacc{\tau_{\cC}^{i-1}=i-1,\tau_{\cC}^i=j}\\
\notag &+(f(\chunk[0]{X}[n-1])-g_{i-1,\pi}(\chunk[0]{X}[i-1]))\indiacc{\tau_{\cC}^{i-1}=i-1,\tau_{\cC}^i\geqslant n-1}|
% \\
% \leqslant 4&\indiacc{\tau_{\cC}^{i-1}=i-1}L\sum_{j=i+1}^{n-1}c_jr^{j-i}+\sum_{j=i}^{n-2}|(g_{j,\pi}(\chunk[0]{X}[j])-g_{i-1,\pi}(\chunk[0]{X}[i-1]))\indiacc{\tau_{\cC}^{i-1}=i-1,\tau_{\cC}^i=j}|\\
% &+|(f(\chunk[0]{X}[n-1])-g_{i-1,\pi}(\chunk[0]{X}[i-1]))\indiacc{\tau_{\cC}^{i-1}=i-1,\tau_{\cC}^i\geqslant n-1}|
\eqsp.
\end{align}
We bound separately all the terms in this decomposition.
First, as $\pi$ is invariant and $f\in \functionboundeddiff{\sX^n}{c}$, for any $j\in \{i+1,\ldots,n-1\}$ and any $\chunk[0]{x}[j]\in \sX^{j+1}$,
\[
|g_{j,\pi}(\chunk[0]{x}[j])-g_{i-1,\pi}(\chunk[0]{x}[i-1])|=\Ebb_{\pi}[f(\chunk[0]{x}[j],\chunk[j+1]{X}[n-1])-f(\chunk[0]{x}[i-1],\chunk[i]{X}[n-1])]\leqslant \sum_{k=i}^jc_k\eqsp.
\]
Hence,
\begin{equation}
\label{eq:Bound:Easyterms}
\begin{aligned} \sum_{j=i}^{n-2}|(g_{j,\pi}(\chunk[0]{X}[j])-g_{i-1,\pi}(\chunk[0]{X}[i-1]))|\indiacc{\tau_{\cC}^i=j}&\leqslant \sum_{j=i}^{n-2}\indiacc{\tau_{\cC}^i=j}\sum_{k=i}^jc_k=\indiacc{\tau_{\cC}^i\leqslant n-2}\sum_{k=i}^{\tau_{\cC}^i}c_k\eqsp,\\
|f(\chunk[0]{X}[n-1])-g_{i-1,\pi}(\chunk[0]{X}[i-1])|\indiacc{\tau_{\cC}^i\geqslant n-1}&\leqslant \indiacc{\tau_{\cC}^i\geqslant n-1}\sum_{k=i}^{n-1}c_k\eqsp.
\end{aligned}
\end{equation}

 %where $|R_{i,1}|$ and $|R_{i,2}|$ are respectively bounded in \eqref{eq:Bound:R1} and \eqref{eq:Bound:R2}.
%\anici{Selon moi, ta majoration de $|R_{i,1}|$ et $|R_{i,2}|$ par le même terme n'était pas juste : j'ai commenté dans le fichier source (au dessus de cette anotation) et propose donc}
%\begin{align*}
% |G_{i,2}-G_{i,1}|=|&R_{i,2}-R_{i,1}+\sum_{j=i}^{n-2}(g_{j,\pi}(\chunk[0]{X}[j])-g_{i-1,\pi}(\chunk[0]{X}[i-1]))\indiacc{\tau_{\cC}^{i-1}=i-1,\tau_{\cC}^i=j}\\
% &+(f(\chunk[0]{X}[n-1])-g_{i-1,\pi}(\chunk[0]{X}[i-1]))\indiacc{\tau_{\cC}^{i-1}=i-1,\tau_{\cC}^i\geqslant n-1}|\\
% \leqslant  |& R_{i,1}| + |R_{i,2}|+\sum_{j=i}^{n-2}|(g_{j,\pi}(\chunk[0]{X}[j])-g_{i-1,\pi}(\chunk[0]{X}[i-1]))\indiacc{\tau_{\cC}^{i-1}=i-1,\tau_{\cC}^i=j}|\\
% &+|(f(\chunk[0]{X}[n-1])-g_{i-1,\pi}(\chunk[0]{X}[i-1]))\indiacc{\tau_{\cC}^{i-1}=i-1,\tau_{\cC}^i\geqslant n-1}|\\
% \leqslant  2&\indiacc{\tau_{\cC}^{i-1}=i-1}L \left( \sum_{j=i+1}^{n-1}c_jr^{j-i} + \sum_{k=\tau_{\cC}^i+1}^{n-1}c_kr^{k-\tau_{\cC}^i} \right)\\
% & +  \sum_{j=i}^{n-2}|(g_{j,\pi}(\chunk[0]{X}[j])-g_{i-1,\pi}(\chunk[0]{X}[i-1]))\indiacc{\tau_{\cC}^{i-1}=i-1,\tau_{\cC}^i=j}|\\
% &+|(f(\chunk[0]{X}[n-1])-g_{i-1,\pi}(\chunk[0]{X}[i-1]))\indiacc{\tau_{\cC}^{i-1}=i-1,\tau_{\cC}^i\geqslant n-1}|\eqsp.
%\end{align*}
To bound $|R_{i,1}|$ and $|R_{i,2}|$ in \eqref{eq:Gi1-Gi2}, we use Lemma~\ref{lem:gi-gipi}.
First, \eqref{eq:gi-gipi} directly yields
\begin{equation}\label{eq:Bound:R1}
|R_{i,1}|\leqslant 2\indiacc{\tau_{\cC}^{i-1}=i-1}L\sum_{j=i+1}^{n-1}c_jr^{j-i}\eqsp.
\end{equation}
Moreover, as $\{\tau_{\cC}^i=j\}\subset \{X_j\in \cC\}$, \eqref{eq:gi-gipi} also yields
\[
(g_j(\chunk[0]{X}[j])-g_{j,\pi}(\chunk[0]{X}[j]))\indiacc{\tau_{\cC}^i=j}\leqslant 2L\sum_{k=j+1}^{n-1}c_kr^{k-j}\indiacc{\tau_{\cC}^i=j}\leqslant 2L\indiacc{\tau_{\cC}^i=j}\sum_{k=\tau_{\cC}^i+1}^{n-1}c_kr^{k-\tau_{\cC}^i}\eqsp.
\]
Therefore,
\begin{equation}\label{eq:Bound:R2}
|R_{i,2}|\leqslant 2L\indiacc{\tau_{\cC}^{i-1}=i-1}\sum_{k=\tau_{\cC}^i+1}^{n-1}c_kr^{k-\tau_{\cC}^i}\eqsp.
\end{equation}
Plugging \eqref{eq:Bound:Easyterms}, \eqref{eq:Bound:R1} and \eqref{eq:Bound:R2} in \eqref{eq:Gi1-Gi2} yields
\begin{align}
\label{eq:Gi1-Gi22} |G_{i,2}-G_{i,1}|\leqslant &2L\bigg(\sum_{j=i+1}^{n-1}c_jr^{j-i}+\sum_{k=\tau_{\cC}^i+1}^{n-1}c_kr^{k-\tau_{\cC}^i}+\frac1{2L}\sum_{k=i}^{\tau_{\cC}^i\wedge (n-1)}c_k\bigg)\indiacc{\tau_{\cC}^{i-1}=i-1}\eqsp.
\end{align}
Both \eqref{eq:Gi-Gi-1} and \eqref{eq:Gi-Gi-1Sq} follow from \eqref{eq:Gi1-Gi22} by bounding separately the $3$ terms in the right-hand side of this inequality.
Let us first establish \eqref{eq:Gi-Gi-1}.
Since $r<1$,
\[
\sum_{j=i+1}^{n-1}c_jr^{j-i}\leqslant \frac{\|c\|_\infty r}{1-r},\qquad \sum_{k=\tau_{\cC}^i+1}^{n-1}c_kr^{k-\tau_{\cC}^i}\leqslant \frac{\|c\|_\infty r}{1-r}\eqsp.
\]
%\anici{Dans toutes les équations suivantes, je pense qu'il y avait une erreur dans l'indice du $\theta$ car par définition $\tau_{\cC}^i = i + \tau_{\cC}^0 \circ \theta^i$ et $\sigma_{\cC}=1+\tau_{\cC}^0 \circ \theta$ d'où $1-i+\tau_{\cC}^i\wedge (n-1) \leqslant 1-i+\tau_{\cC}^i = 1 - i + i +\tau_{\cC}^0 \circ\theta \circ \theta^{i-1}=\sigma_{\cC}\circ\theta^{i-1}$. Cela est par ailleurs cohérent avec \eqref{eq:Gi-Gi-1} que l'on souhaite montrer. J'ai donc modifié en conséquence.}
Moreover,
\[
\sum_{k=i}^{\tau_{\cC}^i\wedge (n-1)}c_k\leqslant \|c\|_{\infty}[1-i+\tau_{\cC}^i\wedge (n-1) ]\leqslant \|c\|_{\infty}[1+\tau_{\cC}^0\circ \theta^i]=\|c\|_{\infty}\sigma_{\cC}\circ\theta^{i-1}\eqsp.
\]
As $r<1\leqslant \sigma_{\cC}\circ\theta^{i-1}$, plugging these upper bounds in \eqref{eq:Gi1-Gi22} shows
\[
|G_{i}-G_{i-1}|=|G_{i,2}-G_{i,1}|\leqslant \frac{5L\|c\|_{\infty}}{1-r}\sigma_{\cC}\circ\theta^{i-1}\indiacc{\tau_{\cC}^{i-1}=i-1}\eqsp.
\]
This proves \eqref{eq:Gi-Gi-1}.
We use slightly different controls to prove \eqref{eq:Gi-Gi-1Sq} from \eqref{eq:Gi1-Gi22}.
As $r\leqslant \rho<1$, $\rho^{-\sigma_{\cC}\circ\theta^{i-1}}\geqslant 1$, and
\begin{align}
\label{eq:Step1Var} \sum_{j=i+1}^{n-1}c_jr^{j-i}\leqslant \sum_{j=i}^{n-1}c_j\rho^{j-i}\leqslant \rho^{-\sigma_{\cC}\circ\theta^{i-1}}\sum_{j=i}^{n-1}c_j\rho^{j-i}\eqsp.
\end{align}
%\anici{Ci après, je pense qu'il y avait également la même erreur d'indice pour $\theta$ et j'ai modifié en conséquence.}
Moreover,
\begin{align*}
 \sum_{k=\tau_{\cC}^i+1}^{n-1}c_kr^{k-\tau_{\cC}^i}\leqslant\rho^{i-\tau_{\cC}^i}\sum_{k=\tau_{\cC}^i+1}^{n-1}c_k\rho^{k-i}\eqsp.
\end{align*}
As $\tau_{\cC}^i\geqslant i$ and $i-\tau_{\cC}^i=1-\sigma_{\cC}\circ\theta^{i-1}$,
\begin{equation}\label{eq:Step2Var}
\sum_{k=\tau_{\cC}^i+1}^{n-1}c_kr^{k-\tau_{\cC}^i}\leqslant \rho^{1-\sigma_{\cC}\circ\theta^{i-1}} \sum_{j=\tau^i_{\cC}+1}^{n-1}c_j\rho^{j-i}\leqslant \rho^{-\sigma_{\cC}\circ\theta^{i-1}}\sum_{j=\tau^i_{\cC}+1}^{n-1}c_j\rho^{j-i}\eqsp.
\end{equation}
In addition,
\begin{align}\label{eq:Step3Var}
 \sum_{k=i}^{\tau_{\cC}^i\wedge (n-1)}c_k\leqslant &\sum_{k=i}^{\tau_{\cC}^i\wedge (n-1)}c_k\rho^{k-\tau_{\cC}^i}=\sum_{k=i}^{\tau_{\cC}^i\wedge (n-1)}c_k\rho^{k-i-\sigma_{\cC}\circ\theta^{i-1}+1}\leqslant \rho^{-\sigma_{\cC}\circ\theta^{i-1}}\sum_{k=i}^{\tau_{\cC}^i\wedge (n-1)}c_k\rho^{k-i}\eqsp.
\end{align}
%\anici{Je ne sais pas si c'est important quantitativement mais dans les deux majorations précédentes, on se débarasses du $\tau_{\cC}^i$ assez "brutalement" alors que si on fait comme dans mon manuscrit de thèse, on parvient dans l'inégalité suivante non pas à $25L^2$ mais $16 L^2$ (en supposant bien évidemment comme ici que $L\geq 1$).}
Plugging \eqref{eq:Step1Var}, \eqref{eq:Step2Var} and \eqref{eq:Step3Var} in \eqref{eq:Gi1-Gi22} and applying Cauchy-Schwarz inequality shows
\begin{align*}
 |G_{i}-G_{i-1}|^2=&|G_{i,2}-G_{i,1}|^2\leqslant16L^2\rho^{-2\sigma_{\cC}\circ\theta^{i-1}}\bigg(\sum_{k=i}^{n-1}c_k\rho^{k-i}\bigg)^2\indiacc{\tau_{\cC}^{i-1}=i-1}\\
 \leqslant &\frac{16L^2}{1-\rho}\rho^{-2\sigma_{\cC}\circ\theta^{i-1}}\sum_{k=i}^{n-1}c_k^2\rho^{k-i}\indiacc{\tau_{\cC}^{i-1}=i-1}\eqsp.
\end{align*}
This proves \eqref{eq:Gi-Gi-1Sq} and thus {\bf Fact 2.}
\end{proof}

\noindent
{\bf Fact 3.} \emph{ Assume {\bf H1}, {\bf H2}, {\bf H3}. For any $x\in \cC$,}
\begin{align}
\Ebb_x\bigg[\rme^{f(\chunk[0]{X}[n-1])-\Ebb_x[f(\chunk[0]{X}[n-1])]}\bigg]\leqslant \rme^{C_3\|c\|^2}\eqsp.
\end{align}
\emph{where $C_3=\fracaa{4L\left(\fracaa{5}{\log u}+4ML\right)}{(1-r\vee u^{-1/4})^2}$.}
\begin{proof}[Proof of Fact 3.]
For any $t\in \Rbb$, $\rme^t\leqslant 1+t+t^2\rme^{|t|}$. Hence, as $\Ebb_{x}[G_{i+1}-G_i|\cF_{\tau_{\cC}^{i}}]=0$, for any $i\geqslant 0$, we have
\begin{align*}
 \Ebb_x[\rme^{G_{i+1}-G_i}|\cF_{\tau_{\cC}^{i}}]\leqslant 1+\Ebb_x[(G_{i+1}-G_i)^2\rme^{|G_{i+1}-G_i|}|\cF_{\tau_{\cC}^{i}}]\eqsp.
\end{align*}
By Fact 2.,
\begin{align*}
 \Ebb_x[\rme^{G_{i+1}-G_i}|\cF_{\tau_{\cC}^{i}}]&\leqslant 1+C_2\sum_{k=i+1}^{n-1}c_k^2\rho^{k-i-1}\indiacc{\tau_{\cC}^{i}=i}\Ebb_x[\rho^{-2\sigma_{\cC}\circ\theta^{i}}\rme^{C_1\|c\|_{\infty}\sigma_{\cC}\circ\theta^{i}}|\cF_{\tau_{\cC}^{i}}]\eqsp.
 \end{align*}
 Now by Markov's property,
\begin{align*}
\indiacc{\tau_{\cC}^{i}=i} \Ebb_x[\rho^{-2\sigma_{\cC}\circ\theta^{i}}\rme^{C_1\|c\|_{\infty}\sigma_{\cC}\circ\theta^{i}}|\cF_{\tau_{\cC}^{i}}]&=\indiacc{\tau_{\cC}^{i}=i} \Ebb_x[\rho^{-2\sigma_{\cC}\circ\theta^{i}}\rme^{C_1\|c\|_{\infty}\sigma_{\cC}\circ\theta^{i}}|\cF_{i}]\\
&=\indiacc{\tau_{\cC}^{i}=i}\Ebb_{X_i}[\rho^{-2\sigma_{\cC}}\rme^{C_1\|c\|_{\infty}\sigma_{\cC}}]\eqsp.
\end{align*}
Hence,
 \begin{align*}
 \Ebb_x[\rme^{G_{i+1}-G_i}|\cF_{\tau_{\cC}^{i}}]
 &=1+C_2\sum_{k=i+1}^{n-1}c_k^2\rho^{k-i-1}\indiacc{\tau_{\cC}^{i}=i}\Ebb_{X_i}[\rho^{-2\sigma_{\cC}}\rme^{C_1\|c\|_{\infty}\sigma_{\cC}}]\eqsp.
 \end{align*}
Let $\rho=r\vee u^{-1/4}$, $\varepsilon=\log u/(2C_1)$ and assume first that $\|c\|_{\infty}\leqslant \varepsilon$. By {\bf H$2$},
\[
\indiacc{\tau_{\cC}^{i}=i}\Ebb_{X_i}[\rho^{-2\sigma_{\cC}}\rme^{C_1\|c\|_{\infty}\sigma_{\cC}}]\leqslant \indiacc{\tau_{\cC}^{i}=i}\sup_{x\in\cC}\Ebb_x[\rho^{-2\sigma_{\cC}}\rme^{C_1\|c\|_{\infty}\sigma_{\cC}}]\leqslant\sup_{x\in\cC}\Ebb_x[u^{\sigma_{\cC}}]\leqslant M\enspace.
\]
Hence,
\[
 \Ebb_x[\rme^{G_{i+1}-G_i}|\cF_{\tau_{\cC}^{i}}]\leqslant 1+C_2M\sum_{k=i+1}^{n-1}c_k^2\rho^{k-i-1}\leqslant \rme^{C_2M\sum_{k=i+1}^{n-1}c_k^2\rho^{k-i-1}}\eqsp.
\]
By recurrence, it follows that
\begin{align*}
\Ebb_x\bigg[\rme^{f(\chunk[0]{X}[n-1])-\Ebb_x[f(\chunk[0]{X}[n-1])]}\bigg]&\leqslant \rme^{C_2M\sum_{i=0}^{n-2}\sum_{k=i+1}^{n-1}c_k^2\rho^{k-i-1}}\\
&=\rme^{C_2M\sum_{k=1}^{n-1}c_k^2\sum_{i=0}^{k-1}\rho^{k-i-1}}\leqslant \rme^{\frac{C_2M}{1-\rho}\|c\|^2} \eqsp.
\end{align*}
Fix $\tilde{x}$ in $\sX$ and let $\tilde{f} : \sX^n \rightarrow \Rbb$ be defined, for any $x_{0:n-1}$ in $\sX^n$, by
\[
\tilde{f}(\chunk[0]{X}[n-1])
=
f(x_0\mathbbm{1}_{\{c_{0}\leq\varepsilon\}} + \tilde{x} \mathbbm{1}_{\{c_{0}>\varepsilon\}},\dots, x_{n-1}\mathbbm{1}_{\{c_{n-1}\leq\varepsilon\}} + \tilde{x} \mathbbm{1}_{\{c_{n-1}>\varepsilon\}})
\eqsp .
\]
As $f$ belongs to $\mathbb{BD}\left(\sX^n,c\right)$, $\tilde{f}$ belongs to $ \mathbb{BD}\left(\sX^n,\tilde{c}\right)$, where
\[
\tilde{c}
=
\left( c_0\mathbbm{1}_{\{c_{0}\leq\varepsilon\}},\dots,c_{n-1}\mathbbm{1}_{\{c_{n-1}\leq\varepsilon\}}\right)
\eqsp .
\]
Since $\|\tilde{c}\|_{\infty} < \varepsilon$ and $\|\tilde{c}\|\leqslant \|c\|$, $\tilde{f}$ satisfies
\begin{align}\label{ineq:max-gamma-leq-varepsilon-ftilde}
\Ebb_{x}\left[\rme^{\tilde{f}(\chunk[0]{X}[n-1])-\Ebb_{x}[\tilde{f}(\chunk[0]{X}[n-1])]}\right]
& \leqslant
\rme^{\frac{MC_2}{1-\rho}
\| \tilde{c}\|^2} \leqslant
\rme^{\frac{MC_2}{1-\rho}
\| c\|^2}
\eqsp .
\end{align}
Furthermore, by definition of $\tilde{f}$ and since $f$ is in $\mathbb{BD}(\sX^n,c)$, for any $x\in\sX^n$,
\begin{align}\label{ineq:difference-f-tildef}
|f(x)-\tilde{f}(x)|
& =
%|f(x_0,\dots,x_{n-1})-f(x_0\mathbbm{1}_{\{\gamma_{0}\leq\varepsilon\}} + \tilde{x} \mathbbm{1}_{\{\gamma_{0}>\varepsilon\}},\dots, x_{n-1}\mathbbm{1}_{\{\gamma_{n-1}\leq\varepsilon\}} + \tilde{x} \mathbbm{1}_{\{\gamma_{n-1}>\varepsilon\}})|\\
%& \leq
\sum_{i=0}^{n-1} c_i \mathbbm{1}_{\{c_i>\varepsilon\}}
\leq
\sum_{i=0}^{n-1} c_i \frac{c_i}{\varepsilon}
\leq
\frac{\|c\|^2}{\varepsilon}.
\end{align}
This implies
\[
\Ebb_x\bigg[\rme^{f(\chunk[0]{X}[n-1])-\Ebb_x[f(\chunk[0]{X}[n-1])]}\bigg]\leqslant \rme^{\frac{2\|c\|^2}{\varepsilon}}\Ebb_{x}\left[\rme^{\tilde{f}(\chunk[0]{X}[n-1])-\Ebb_{x}[\tilde{f}(\chunk[0]{X}[n-1])]}\right]\leqslant \rme^{\bigg(\frac2{\varepsilon}+\frac{MC_2}{1-\rho}\bigg)
\| c\|^2}
\eqsp .
\]
This shows {\bf Fact 3} since
\[
\frac2{\varepsilon}+\frac{MC_2}{1-\rho}\leqslant \frac{4L}{(1-r\vee u^{-1/4})^2}\bigg(\frac5{\log u}+4ML\bigg)\eqsp.
\]
\end{proof}

\noindent
{\bf Fact 3} proves that there exists a constant $C=2C_3$ such that, for any $c\in \Rbb^n$, $f\in \functionboundeddiff{\sX^n}{c}$ and $x\in \cC$,
\begin{equation}\label{ControlLaplace}
 \Ebb_x\bigg[\rme^{f(\chunk[0]{X}[n-1])-\Ebb_x[f(\chunk[0]{X}[n-1])]}\bigg]\leqslant \rme^{C\| c\|^2/2}
\eqsp .
\end{equation}
Let $f\in \functionboundeddiff{\sX^n}{c}$ and $x\in \cC$. 
For any $s>0$, $sf\in \functionboundeddiff{\sX^n}{c}$.
Hence, from \eqref{ControlLaplace}, for any $s,t>0$,
\begin{align*}
 \Pbb\big(f(\chunk[0]{X}[n-1])-\Ebb_x[f(\chunk[0]{X}[n-1])]>t\big)&\leqslant \rme^{-st+\log\Ebb_x\big[\rme^{sf(\chunk[0]{X}[n-1])-\Ebb_x[sf(\chunk[0]{X}[n-1])]}\big]}\\
 &\leqslant \rme^{-st+s^2C\| c\|^2/2}\enspace.
\end{align*}
Choosing $s=t/(C\|c\|^2)$ proves Theorem~\ref{thm:ConcMarkQuant} with
\[
\beta=\frac1{2C}=\frac1{4C_3}=\frac{(1-r\vee u^{-1/4})^2}{16L}\bigg(\frac5{\log u}+4ML\bigg)^{-1}\eqsp.
\]
\end{proof}

\bibliographystyle{plain}
\bibliography{BiblioSource}

\begin{thebibliography}{10}

\bibitem{MR2424985}
R.~Adamczak.
\newblock A tail inequality for suprema of unbounded empirical processes with
  applications to {M}arkov chains.
\newblock {\em Electron. J. Probab.}, 13:no. 34, 1000--1034, 2008.

\bibitem{MR3407208}
J.~Dedecker and S.~Gou\"{e}zel.
\newblock Subgaussian concentration inequalities for geometrically ergodic
  {M}arkov chains.
\newblock {\em Electron. Commun. Probab.}, 20:no. 64, 12, 2015.

\bibitem{douc:moulines:priouret:soulier:2018}
R.~Douc, E.~Moulines, P.~Priouret, and P.~Soulier.
\newblock {\em {M}arkov chains}.
\newblock Springer, 2018.

\bibitem{ghosal:vandervaart:2017}
S.~Ghosal and A.~van~der Vaart.
\newblock {\em Fundamentals of Nonparametric {B}ayesian Inference}.
\newblock Cambridge Series in Statistical and Probabilistic Mathematics.
  Cambridge University Press, 2017.

\bibitem{lecorff:lerasle:vernet:2018}
S.~{Le Corff}, M.~{Lerasle}, and E.~{Vernet}.
\newblock {A {B}ayesian nonparametric approach for generalized Bradley-Terry
  models in random environment}.
\newblock arXiv:1808.08104, 2018.

\bibitem{MR1036755}
C.~McDiarmid.
\newblock On the method of bounded differences.
\newblock In {\em Surveys in combinatorics, 1989 ({N}orwich, 1989)}, volume 141
  of {\em London Math. Soc. Lecture Note Ser.}, pages 148--188. Cambridge Univ.
  Press, Cambridge, 1989.

\bibitem{paulin:2015}
D.~Paulin.
\newblock Concentration inequalities for {M}arkov chains by {M}arton couplings
  and spectral methods.
\newblock {\em Electronic Journal of Probability}, 20:1--32, 2015.

\bibitem{MR2095565}
G.~O. Roberts and J.~S. Rosenthal.
\newblock General state space {M}arkov chains and {MCMC} algorithms.
\newblock {\em Probab. Surv.}, 1:20--71, 2004.

\bibitem{rousseau:2016}
J.~{Rousseau}.
\newblock On the frequentist properties of {B}ayesian nonparametric methods.
\newblock {\em Annual Review of Statistics and Its Application}, 3(1):211--231,
  2016.

\bibitem{vernet:2015}
E.~Vernet.
\newblock Posterior consistency for nonparametric hidden {M}arkov models with
  finite state space.
\newblock {\em Electronic Journal of Statistics}, 9(1):717--752, 2015.

\end{thebibliography}

\end{document}